\documentclass[a4paper,oneside, 12pt]{article}

\usepackage{amsmath, amsthm, amsfonts, amssymb}
\usepackage[all]{xy}
\usepackage{graphicx}
\usepackage{MnSymbol}
\usepackage{tikz}
\usepackage{hyperref}
\usepackage[nottoc]{tocbibind}
\usetikzlibrary{matrix, arrows,positioning}
\usepackage{float}
\usepackage[ruled, vlined]{algorithm2e}

\numberwithin{equation}{section}

\newtheorem{theorem}{Theorem}[section]
\newtheorem{corollary}[theorem]{Corollary}
\newtheorem{lemma}[theorem]{Lemma}
\newtheorem{proposition}[theorem]{Proposition}
\newtheorem{conjecture}[theorem]{Conjecture}

\theoremstyle{definition}

\newcommand{\C}{\mathbb{C}}
\newcommand{\R}{\mathbb R}
\newcommand{\Q}{\mathbb Q}
\newcommand{\Z}{\mathbb Z}
\newcommand{\p}{\mathbb P}

\newcommand{\xto}{\xrightarrow}
\newcommand{\action}{\curvearrowright}

\def\url#1{\noindent \sf{#1}}

\begin{document}

\title{\bf Homology supported in Lagrangian submanifolds in mirror quintic threefolds}

\author{Daniel L\'opez Garcia}

\date{}

\maketitle

\begin{abstract}
In this note, we study homology classes in the mirror quintic Calabi-Yau threefold which   can be realized by  special Lagrangian submanifolds. We have used Picard-Lefschetz  theory to establish the monodromy action  and to study the orbit of Lagrangian vanishing cycles. For many prime numbers $p$ we can compute the orbit modulo $p$.  We conjecture that the orbit in  homology with coefficients in $\mathbb{Z}$ can be determined by these orbits with coefficients in $\mathbb{Z}_p$.
\end{abstract}



\section{Introduction}
Given a symplectic manifold $(X, \omega)$ of dimension $2n$ there are homology classes in $H_n(X, \Q)$ which may be represented by Lagrangian cycles.  In \cite{minarea}, the authors define Lagrangian cycles as cycles in a symplectic 4-manifold, whose two-simplices are given by $C^1$ Lagrangian maps and a \textit{Lagrangian homology class} is a homology class which can be represented by a Lagrangian cycle. In that article, they  show a characterization of the  Lagrangian homology classes in terms of the minimizers of an area functional. Moreover, they show for a compact K\"alher 4-manifold $(X, \omega, J)$ and a homology class $\alpha\in H_2(X,\Z)$, that $\alpha$ is a Lagrangian homology class if and only if $[\omega](\alpha)=0$. If the Chern class $c_1(X)$ also annihilates $\alpha$, then $\alpha$ can be represented by an immersed Lagrangian surface (not necessarily embedded). 

The question about which part of the homology is supported in Lagrangian submanifolds, can be refined a little more if we look for Lagrangian spheres. In \cite{langrangianspheres} for a $(X, \omega, J)$ K\"alher 4-manifold with Kodaira dimension $-\infty$, i.e. for rational or ruled surfaces it is shown that the class $\alpha\in H_2(X,\Z)$ is represented by a Lagrangian sphere if and only if $[\omega](\alpha)=0$, $c_1(X)(\alpha)=0$, $\alpha^2=-2$ and  $\alpha$ is represented by a smooth sphere. For 4-manifolds, the dimension of the 2-cycles  allows us to relate the property of being represented by Lagrangian cycles with the vanishing of the \textit{periods} $\int_{\alpha} \omega$ and $\int_{\alpha}c_1(X)$. For higher dimension manifolds this pairing is not well-defined, hence we do not have a natural generalization of the previous results. Despite of this, it is possible to show that in any regular hypersurface of $\p^n$ with $n$ even, all $(n-1)$-cycles can be written as a linear combination of cycles supported in Lagrangian spheres, see Proposition \ref{main1}.

A more interesting question for $n=4$, is to ask not only which homology classes  are generated by Lagrangian spheres but which ones are supported in Lagrangian spheres. In this article we consider a family $\tilde{X}_\varphi$ of mirror quintic Calabi-Yau threefolds and study some classes in $H_3(\tilde{X}_\varphi,\Z)$ which are supported in Lagrangian $3$-spheres and Lagrangian $3$-tori. This family is constructed as follows. Consider the Dwork family $X_{\varphi}$ in $\p^4$ given by the locus of the polynomial  
$$p_{\varphi}:=\varphi z_0^5+z_1^5+z_2^5+z_3^5+z_4^5-5 z_0z_1z_2z_3z_4=0,$$
with critical values  in $\varphi=0,1, \infty$. For every $\varphi\neq 0,1,\infty$, $\tilde{X}_\varphi$ is obtained as a desingularization of the quotient of $X_\varphi$ by the action of a finite group, see \S \ref{monodromyonMQT} and \cite{candelasdelaossa, PFmonodromy, doranmorgan}. The rank of the free group $H_3(\tilde X_{\varphi}, \Z)$ is four and hence it is isomorphic to $\Z^4$ after choosing a basis. In this basis the homology class $\delta_2=(0 \text{ }1\text{ }0\text{ }0)$ is represented by a torus  associated to the singularity of $X_\varphi$ when $\varphi \to 0$ and the class $\delta_4=(0\text{ }0\text{ }0\text{ }1)$ is represented by a sphere $S^3$ associated to the singularity of $X_\varphi$ when $\varphi\to 1$.  As in \cite{candelasdelaossa} we give an explicit description of these two cycles in \S \ref{lagrangianST}, and furthermore we show that these cycles are Lagrangian submanifolds of $\tilde X_{\varphi}$.

The monodromy action of the family is given by symplectomorphisms at each regular fiber. It is possible to determine two matrices $M_0$ and $M_1$ such that the monodromy action over $H_3(\tilde{X}_\varphi,\Z)$ corresponds (with respect to the basis mentioned above) to the free subgroup of $Sp(4,\Z)$ generated by $M_0$ and $M_1$, see \S \ref{monodromyonMQT}.
Therefore, the orbit of $\delta_2$ and $\delta_4$ by the action of $M_0*M_1$ are homology classes which can be represented by   Lagrangian submanifolds. Our main result is about $H_3(\tilde X_{\varphi}, \Z_p)$, where $\Z_p=\Z/ p\Z$ for some primes $p$, and it is summarized in the following theorem.
\begin{theorem}
	\label{main2}
	For the mirror quintic Calabi-Yau threefold  $\tilde X:=\tilde X_{\varphi}$ with $\varphi\neq 0,1,\infty$, the homology classes
	\begin{align}
	    \label{cycle2torus}
	(0\text{ }0\text{ }1\text{ }1),\text{ } 	
	(0\text{ }1\text{ }0\text{ } 0),\text{ }
	(0\text{ }1\text{ }0\text{ }1),\text{ }
	(1\text{ }     0\text{ }     0\text{ }     1),\text{ }
	(1\text{ }0\text{ }     1\text{ }     1) & \in H_3(\tilde X, \Z_2)
	\end{align}
	\begin{align}
	\label{cycle5torus}
	(0\text{ }1\text{ }0\text{ }0),\text{}
	(0\text{ }1\text{ }0\text{ }1),\text{}
	(0\text{ }1\text{ }0\text{ }2),\text{}
	(0\text{ }1\text{ }0\text{ }3),\text{}
	(0\text{ }1\text{ }0\text{ }4),\text{}
	\nonumber\\
	(0\text{ }1\text{ }1\text{ }0),\text{}
	(0\text{ }1\text{ }1\text{ }1),\text{}
	(0\text{ }1\text{ }1\text{ }2),\text{}
	(0\text{ }1\text{ }1\text{ }3),\text{}
	(0\text{ }1\text{ }1\text{ }4),\text{}
	\nonumber\\
	(0\text{ }1\text{ }2\text{ }0),\text{}
	(0\text{ }1\text{ }2\text{ }1),\text{}
	(0\text{ }1\text{ }2\text{ }2),\text{}
	(0\text{ }1\text{ }2\text{ }3),\text{}
	(0\text{ }1\text{ }2\text{ }4),\text{}
	\nonumber\\
	(0\text{ }1\text{ }3\text{ }0),\text{}
	(0\text{ }1\text{ }3\text{ }1),\text{}
	(0\text{ }1\text{ }3\text{ }2),\text{}
	(0\text{ }1\text{ }3\text{ }3),\text{}
	(0\text{ }1\text{ }3\text{ }4),\text{}
	\nonumber\\
	(0\text{ }1\text{ }4\text{ }0),\text{}
	(0\text{ }1\text{ }4\text{ }1),\text{}
	(0\text{ }1\text{ }4\text{ }2),\text{}
	(0\text{ }1\text{ }4\text{ }3),\text{}
	(0\text{ }1\text{ }4\text{ }4) \text{}
	& \in H_{3}(\tilde X, \Z_5)
	\end{align} are represented by Lagrangian 3-tori. 	The homology classes
    \begin{align}
	\label{cycle2sphere}	(0\text{ }0\text{ }0\text{ }1),\text{ } 
	(0\text{ }0\text{ }1\text{ }0),\text{} 	
	(0\text{ }1\text{ }1\text{ }0),\text{}
	(0\text{ }1\text{ }1\text{ }1),\text{}
	(1\text{ }0\text{ }0\text{ }0),\text{}
	\nonumber\\
	(1\text{ }0\text{ }1\text{ }0),\text{}
	(1\text{ }1\text{ }0\text{ }0),\text{}
	(1\text{ }1\text{ }0\text{ }1),\text{}
	(1\text{ }1\text{ }1\text{ }0),\text{}
	(1\text{ }1\text{ }1\text{ }1)  &\in H_3(\tilde X, \Z_2)
	\end{align}
	\begin{equation}
	\label{cycle5sphere}
	(0 \text{ }0\text{ }0\text{ }1),\text{ }
	(0 \text{ }0\text{ }1\text{ }1),\text{ }
	(0 \text{ }0\text{ }2\text{ }1),\text{ }
	(0 \text{ }0\text{ }3\text{ }1),\text{ }
	(0 \text{ }0\text{ }4\text{ }1)\text{ } \in H_3(\tilde X, \Z_5)
	\end{equation}
	are represented by Lagrangian 3-spheres. 	For $p=3,7,11,13,17,19,23$, any homology class in $H_3(\tilde X, \Z_p)$ different from $(0\text{ }0\text{ }0\text{ }0)$ can be represented by Lagrangian 3-tori and by  Lagrangian 3-spheres.
\end{theorem}
In general for a manifold $M$, a class $\delta\in H_k(M,\Z)$ is called \textit{primitive} if there is no $m\in \Z$ and $\delta '\in H_k(M, \Z) $ such that $\delta=m \delta '$.  We believe that for any prime different to $2$ and $5$, all classes in $H_3(\tilde X, \Z_p)$ different to $(0\text{ }0\text{ }0\text{ }0)$ can be represented by  Lagrangian 3-tori and by a Lagrangian 3-spheres. This is a consequence of the following conjecture.
\begin{conjecture}
	\label{conjecture1}
	Let  $\delta$ be a primitive class in $ H_3(\tilde X, \Z)$. If $mod_2(\delta)$ is a homology class in the list (\ref{cycle2torus})  and $mod_5(\delta)$ is a homology class in the list (\ref{cycle5torus}), then $\delta$  is represented by a  Lagrangian 3-torus. If $mod_2(\delta)$ is a homology class in the list (\ref{cycle2sphere}) and $mod_5(\delta)$ is a homology class in the list (\ref{cycle5sphere}), then $\delta$ is represented by a Lagrangian 3-sphere.
\end{conjecture}

We have analogous results for other 14 examples of Calabi-Yau threefolds which appear in Table \ref{14examplevalues}. However, in these cases we do not know if the vectors $\delta_2$ and $\delta_4$ have Lagrangian submanifolds associated as in the Dwork family case.\\

\textbf{Acknowledgements}. I thank Hossein Movasati  for suggesting to study monodromy action of quintic mirror and for helpful discussions.  Also a discussion with Emanuel Scheidegger about the Lagrangian torus was very useful. Furthermore, I would like to thank Roberto Villaflor for his comments and suggestions on a first version of the paper.
\section{Basics on Picard-Lefschetz theory}
We recall some facts about Lefschetz fibration in symplectic geometry. These results are in the literature, see for example \cite{ArouxLagrangian, Arouxlefschetz,seidel, RThomas}. We collect them here to set notations and for quick reference throughout the article.\\

Let $Y$ be a complex manifold. A  \textit{Lefschetz fibration} is a surjective analytic map $f:Y\rightarrow\mathbb{P}^1$ with a finite number of critical points, such that for any critical point $p$, there is a  chart with Morse coordinates. This means that there is a coordinate system around $p$ such that $f(z)=f(p)+z_1^2+\cdots+z_n^2$ for $z$ in a neighborhood of $p$.

Every projective manifold $Y\hookrightarrow \mathbb{P}^N$ has a natural symplectic form $\omega$ given by the pullback of the Fubini study form in $\mathbb{P}^N$. Since the fibers of $f$ over regular values are complex submanifolds of $Y$, the restriction of $\omega$ to each regular fiber remains symplectic. Furthermore, the regular fibers of the Lefschetz fibration $f:Y\rightarrow \mathbb{P}^1$ are symplectomorphic. This follows from the following symplectic version of the Ehresmann lemma. 
\begin{proposition}
	\label{symplecticEhresmann}
	Let $(E, \omega)$ be a symplectic manifold and $B$ be a connected manifold. Consider $f:E\rightarrow B$ a proper surjective map with a finite set of critical values $C$, such that $\omega$ is symplectic at every regular fiber of $f$. Then the regular fibers are symplectomorphic.
\end{proposition}

\begin{proof}
	Using $\omega$ we can decompose the tangent bundle $TE$, over the set of regular values, as a direct sum of a vertical bundle $VE$ and a horizontal bundle $HE:=(VE)^\omega$. Here, the vertical space $V_eE$ is the space of vectors tangent to the fibers of $f$ and the horizontal space $H_eE$ is its symplectic complement. This is well-defined since the restriction of $\omega$ to the fibers is symplectic.
	 
	 Let $b\in B$ be a regular value and $ U\subset B\setminus C$ be a neighborhood of $b$. We take a  vector field $W$ defined on $U$, without singularities. Since $f$ is a submersion on $U$, the map $f_*$ is an isomorphism between $H_eE$ and $T_{f(e)}B$ for all $f(e)\in U$
	 , thus  we can take the vector field $V:=f^*W$ on $E_U$.	 Because the fibers are compact, the flow $\theta$  of $V$ is defined in a neighborhood of $E_b$ for all $t$ in some interval $I$. Therefore $\varphi_t:=\theta(-,t)$ is a diffeomorphism between $E_b$ and some other fiber in a neighborhood. 
	
	In order to show that $\varphi_t$ preserves the symplectic form at the fibers is enough to show that $\left.\frac{d}{d\tau}\right\rvert_{\tau=t}\varphi_\tau^*\omega_{b}=0$ for $t \in I$, where   $\omega_{b}$ is the form $\omega$ restricted to the fiber $E_{b}$.  This follows noting that
	\begin{equation*}
	\left.\frac{d}{d\tau}\right\rvert_{\tau=t}\varphi_\tau^*\omega_{b}=\varphi_t^*(\mathcal{L}_{V}\omega_{b})=\varphi_t^*(d\imath_{V}\omega_{b}+\imath_{V}d\omega_{b})=\varphi_t^*(d\imath_{V}\omega_{b}),
	\end{equation*}
	and that $\imath_{V}\omega_{b}=0$ since $V$ is in $HE$.
\end{proof}
Let $\gamma:[0,1]\to B\setminus C$ be a simple path. We denote by $P_\gamma:E_{\gamma(0)}\rightarrow E_{\gamma(1)}$ the symplectomorphism given by the lifting of $\gamma$ as in the previous proposition.
\begin{corollary}
	Let $Y$ be a projective manifold and  $f:Y\rightarrow \mathbb{P}^1$ be a Lefschetz fibration with critical values $C$. For a simple path $\gamma:[0,1]\to \mathbb{P}^1\setminus C$, the map $P_\gamma$ is a symplectomorphism.
\end{corollary}
  From now on, we will denote by $X:=Y_b\hookrightarrow Y$ a regular fiber of $f$. Since these fibers are  symplectomorphic we simply denote any symplectic fiber by $(X, \omega_X)$. Thus,  we have a map $\pi_1(\mathbb{P}^1\setminus C)\rightarrow {\rm Symp}(X,\omega_X)$ which descends to homology, inducing the so called \textit{monodromy action} $\pi_1(\mathbb{P}^1\setminus C)\action H_*(X, \Z)$ given by $(\gamma, \delta) \to (P_\gamma)_*\delta$.

Let $\gamma:[0,1]\rightarrow \mathbb{P}^1$ be a simple path such that $\gamma(1)\in C$ and $\gamma(t)\in\mathbb{P}^1\setminus C$ for $t\in [0,1)$. Let $p$ be a critical point in $f^{-1}(\gamma(1))$. 	The  set of points $$ 
	V_{\gamma}=\{z\in f^{-1}(Im(\gamma))\text{ }|\text{ }\lim_{t\rightarrow 1}P_{\gamma}(t)(z)=p\}  $$
is called \textit{Lefschetz thimble} and the intersection of $V_{\gamma}$ with the fiber $f^{-1}(\gamma(0))$  is called \textit{vanishing cycle} $\delta_{\gamma}$.
\begin{proposition}
	\label{vanishingthimble}
	The Lefschetz thimble $V_\gamma$ is a Lagrangian submanifold of $(Y, \omega)$ and the vanishing cycle $\delta_\gamma$ is a Lagrangian sphere of $(X, \omega_X)$.
\end{proposition}

\begin{proof}
In a  compact neighborhood $U$ of $p$ we can  suppose that $f(z)=f(p)+z_1^2+\cdots+z_n^2$ and that $\gamma$ is a real curve in $\C$ with $\gamma(1)=0$, and $\gamma(0)>0$. 
Let  $H:U\to \R$ be the map given by $H(z)=\text{Re} (f(z))$. The Hamiltonian vector field $X_H$ is horizontal  because $H$ is constant in the fibers of $f$ and $\omega(X_H, V)=dH(V)=0$ for any vertical vector field $V$. Since $JV$ is also vertical then $\nabla H$ is horizontal. On the other hand $-\nabla H$ projects to $\frac{\partial }{\partial x}$ and so $V_\gamma$ is the unstable set of $p$ .
	
By a direct computation $H$ is a Morse function with index $n$. Using the unstable manifold theorem \cite[Thm. 4.2]{stable} we conclude that $V_\gamma$  is a $n$-ball  inside $Y$. To see that $V_{\gamma}$ is isotropic,  consider $u,v\in T_zV_{\gamma}$ for any $z\in V_{\gamma}$. Since the horizontal component of $V_\gamma$ is one dimensional,  we have 
$	\omega_z(u,v)=\omega_X(z)(u_v,v_v),$  where $u_v$ and $v_v$ are the vertical components of $u$ and $v$. As the fibers over $\gamma(t)$ with $t\in[0,1)$ are symplectomorphic via $\varphi_t$, we have that 
$$\omega_X(z)(u_v, v_v)=\omega_X({z(t)})(u_v(t), v_v(t))$$ where $z(t)=\varphi_t(z)$, $u_v(t)=(\varphi_t(z))_{*}u_v$  and  $v_v(t)=(\varphi_t(z))_{*}v_v$. In  the limit the tangent space is a point, then by continuity we can conclude the result.
\end{proof}
Let $f\in \C[z_1,\ldots, z_n]$ and let $b\in\C$ be some  regular value of $f$. Suppose that the origin is an isolated critical point of the highest-grade homogeneous piece of $f$. The $(n-1)$-homology group of the fiber over $b$ is generated by the vanishing cycles, see \cite{Lamotke}, \cite[\S 7.4]{hodgehossein}. As a consequence we can prove the next proposition.
\begin{proposition}
	\label{main1}
	Let $F\in \C[z_0,\ldots, z_{n}]$ be a homogeneous polynomial with  $n$ even. Suppose that $F$ defines  a smooth variety $X$ in $\p^n$. Then, any homology class $\delta\in H_{n-1}(X,\Z)$ can be written as a finite sum $\delta=\sum_{j} a_j\delta_j$, where $a_j\in \Z$ and $\delta_j$ is supported in a Lagrangian $(n-1)$-sphere. 
\end{proposition}

\begin{proof}
	Consider a hyperplane that intersects transversally $X$, and let $Z$ be  its intersection. We can suppose that the hyperplane section is $Z=X\cap \{z_{0}=0\}$. Let $f\in \C[z_1,\ldots, z_n]$ be the  polynomial $F(1,z_1,\ldots, z_n)$ and we define the affine variety $U:=X\setminus Z=\{(z_1,\ldots, z_n)\in \C^{n}\text{ }|\text{ }f(z_1,\ldots, z_n)=0\}$.  The pair $(X, U)$ induces the exact sequence in homology $$\cdots \to H_{n}(X,U)\to H_{n-1}(U)\to H_{n-1}(X)\to H_{n-1}(X,U)\to \cdots,$$ where the map $H_k(U)\to H_k(X)$ comes from  the inclusion $U\subset X$. By Leray-Thom-Gysin isomorphism we have $H_{k}(X,U)\simeq H_{k-2}(Z)$. By Lefschetz hyperplane section theorem we know that $H_k(Z)\simeq H_k(\p^{n-2})$ if $k\neq n-2$, see \cite[\S 5.4]{hodgehossein}. Since $n$ is even we have that $H_{n-3}(Z)=0$, then the map $$H_{n-1}(U)\to H_{n-1}(X)\to H_{n-3}(Z)=0$$is surjective. The vanishing cycles associated to the fibration $f:\C^n\to \C$ generate the homology group $H_{n-1}(U)$, and they are supported in Lagrangian spheres of $U$.
\end{proof}

\section{Monodromy action on mirror quintic threefolds}
\label{monodromyonMQT}
In this section we recall the definition of a mirror quintic Calabi-Yau threefold and its monodromy action coming from the Picard-Fuchs equations. We also list the monodromy action of other 14 examples of Calabi-Yau threefolds.  For a more detailed description, the reader is referred to \cite{candelasdelaossa, doranmorgan, movasatigaussconnection, HMSUeda}.\\

The family of hypersurfaces in $\p^4$ given by a generic polynomial of degree 5 is denoted $\p^4[5]$. The elements of $\p^4[5]$ are quintic Calabi-Yau threefolds, with Hodge numbers $h^{1,1}=1$ and $h^{2,1}=101$. Let $\{X_\varphi\}_{\varphi}$ be the one-parameter family of hypersurfaces in $\p^4$ given by
\begin{equation}
\label{pz}
p_\varphi=\varphi z^5_0+z^5_1+z^5_2+z^5_3+z^5_4-5 z_0z_1z_2z_3z_4,\hspace{4mm}\varphi\neq 0,1.
\end{equation}
Consider the finite group
$$G=\{(\xi_0,\xi_1, \xi_2, \xi_3, \xi_4)\in \C^5\text{ }|\text{ }\xi_i^5=1\text{ , }\xi_0\xi_1\xi_2\xi_3\xi_4=1\}$$
acting on $\p^4$, as $(\xi_0,\xi_1, ..., \xi_4)\cdot [z_0:z_1:\cdots:z_4]=[\xi_0z_0:\cdots:\xi_4z_4]$. It is known that the action of $G$ is free away from  the curves $C_{ijk}:=\{z_i^5+z_j^5+z_k^5=0, \text{ }z_l=0 \text{ for all } l\neq i,j,k\}$ for $0\le i<j<k\le 4$, see \cite{MorrisonMirror}.  The \textit{mirror quintic Calabi-Yau threefold}, mirror quintic for short, is the variety $\tilde{X}_\varphi$  obtained after resolving the orbifolds singularities of the quotient $X_\varphi/G$. The manifold $\tilde X_\varphi$, has Hodge numbers $h^{1,1}=101$ and $h^{2,1}=1$ and Betti number $b_3=4$. In terms of the mirror symmetry $\p^4[5]$ is called the \textit{A-model} and $\{\tilde X_{\varphi}\}_{\varphi}$ the \textit{B-model}, see for example \cite{KontsevichHAMS}.

The variety $\tilde X_\varphi$  has a  holomorphic 3 form $\eta$  that vanishes nowhere. Moreover, $H^{3,0}$ is spanned by $\eta$.  The \textit{periods} of $\eta$ are functions $\int_\Delta \eta$, where the homology class $\delta=[\Delta]\in H_3(\tilde X_\varphi, \Z)$ is supported in the submanifold $\Delta$. The fourth-order linear differential equation $$\left(\theta^4-\varphi\left(\theta+\frac{1}{5}\right)\left(\theta+\frac{2}{5}\right)\left(\theta+\frac{3}{5}\right)\left(\theta+\frac{4}{5}\right)\right)y=0, \hspace{4mm}\theta=\varphi\frac{\partial}{\partial \varphi}$$ 
is called Picard-Fuchs equations, and its solutions are the periods of $\eta$.

The Picard-Fuchs ODE has 3 regular singular points $\varphi=0, 1, \infty$. The analytic continuation of this ODE, gives us the monodromy operators $M_0, M_1, M_{\infty}$. Since the monodromy is a representation $\rho:\pi_1(\p^1\setminus\{0,1,\infty\})\to Sp(4)$, we have the relation $M_0M_1M_\infty=Id$. There exits a basis
 such that the monodromy operators in this basis are written as
$$M_0=\begin{pmatrix}
1&1&0&0\\
0&1&0&0\\
5&5&1&0\\
0&-5&-1&1
\end{pmatrix} \text{ and } M_1=\begin{pmatrix}
1&0&0&0\\
0&1&0&1\\
0&0&1&0\\
0&0&0&1
\end{pmatrix},$$
see for example \cite{pingpong,  PFmonodromy, doranmorgan}.

The matrices $M_0$ and $M_1$ are conjugated to the matrices
$$T_0=\begin{pmatrix}
1&1&0&0\\
0&1&5&0\\
0&0&1&1\\
0&0&0&1
\end{pmatrix} \text{ and } T_1=\begin{pmatrix}
1&0&0&0\\
-5&1&0&0\\
-1&0&1&0\\
-1&0&0&1
\end{pmatrix}$$
appearing in \cite{PFmonodromy,doranmorgan}, via the matrix 
$P=\begin{pmatrix}
0&0&0&-1\\
0&5&1&0\\
1&1&0&0\\
0&1&0&0
\end{pmatrix}.
$
Thus $P^{-1}T_iP=M_i$, $i=0,1$. In \cite{candelasdelaossa}, the matrices for the monodromy are
$$S_\infty=\begin{pmatrix}
51&90&-25&0\\
0&1&0&0\\
100&175&-49&0\\
-75&-125&35&1
\end{pmatrix} \text{ and } S_1=\begin{pmatrix}
1&0&0&0\\
0&1&0&1\\
0&0&1&0\\
0&0&0&1
\end{pmatrix}$$
these matrices are associated to the equation 
\begin{equation*}
\label{ppsi}
p_{\psi}=z_0^5+z_1^5+z_2^5+z_3^5+z_4^5-5\psi z_0z_1z_2z_3z_4,
\end{equation*} with singularities at $\psi^5=1, \infty$. The change of variable $\psi=\varphi^{\frac{-1}{5}}$ gives us the family defined by the equation (\ref{pz}). Moreover, the matrix $M_0^5$ is conjugated to $S_\infty$. In fact with the matrix 
\begin{equation}
\label{changeofbasis}
M=\begin{pmatrix}
3&0&1&0\\
0&1&0&0\\
5&0&2&0\\
0&0&0&1
\end{pmatrix}
\end{equation}
we obtain the equations $M^{-1}S_1M=M_1$ and $M^{-1}S_\infty M=M_0^5$.\\

More generally, it is known that the differential equation
\begin{equation}
\label{CY3ODE}
(\theta^4-\varphi(\theta+A)(\theta+1-A)(\theta+B)(\theta+1-B))y=0,\hspace{4mm}\theta=\varphi\frac{\partial}{\partial \varphi}
\end{equation}
corresponds to the Picard-Fuchs equation of a mirror Calabi-Yau threefold for  14 values of $(A, B)$, and the singularities are in $\varphi=0,1,\infty$. We have listed the \textit{A-model} of these 14 examples in   Table \ref{14examplevalues}.
\begin{table}[H]
	\small
	\label{table14example}
	\centering
	\begin{tabular}{|l|l|l|l|}
		\hline
		$(d,k)$&$A$ & $B$ & A-model of equation (\ref{CY3ODE}) \\
		\hline
		\hline
		$(5,5)$&$1/5$&$2/5$&$X(5)\subset \p^4$\\
		\hline
		$(2,4)$&$1/8$&$3/8$&$X(8)\subset \p^4(1,1,1,1,4)$\\
		\hline
		$(1,4)$&$1/12$&$5/12$&$X(2,12)\subset \p^5(1,1,1,1,4,6)$\\
		\hline
		$(16,8)$&$1/2$&$1/2$&$X(2,2,2,2)\subset \p^7$\\
		\hline
		$(12,7)$&$1/3$&$1/2$&$X(2,2,3)\subset \p^6$\\
		\hline
		$(8,6)$&$1/4$&$1/2$&$X(2,4)\subset \p^5$\\
		\hline
		$(4,5)$&$1/6$&$1/2$&$X(2,6)\subset \p^5(1,1,1,1,1,3)$\\
		\hline
		$(2,3)$&$1/4$&$1/3$&$X(4,6)\subset \p^5(1,1,1,2,2,3)$\\
		\hline
		$(1,2)$&$1/6$&$1/6$&$X(6,6)\subset \p^5(1,1,2,2,3,3)$\\
		\hline
		$(6,5)$&$1/6$&$1/4$&$X(3,4)\subset \p^5(1,1,1,1,1,2)$\\
		\hline
		$(3,4)$&$1/6$&$1/3$&$X(6)\subset \p^4(1,1,1,1,2)$\\
		\hline
		$(1,3)$&$1/10$&$3/10$&$X(5)\subset \p^4(1,1,1,2,5)$\\
		\hline
		$(4,4)$&$1/4$&$1/4$&$X(4,4)\subset \p^5(1,1,1,1,2,2)$\\
		\hline
		$(9,6)$&$1/3$&$1/3$&$X(3,3)\subset \p^5$\\
		\hline
	\end{tabular} 
	\caption{\footnotesize Fourteen values for equation \ref{CY3ODE} with the corresponding Calabi-Yau threefold.}
	\label{14examplevalues}
\end{table}

The notation $X(d_1,d_2,\ldots, d_l)\subset \p^n(w_1,w_2,\ldots, w_n)$ denotes a complete intersection of $l$ hypersurfaces of degrees $d_1, d_2,\ldots, d_l$ in the weighted projective space with weight $(w_1, w_2,\ldots, w_n)$, see for example \cite{PFmonodromy}. For these cases the monodromy matrices correspond to the same $M_1$ as before and $$M_0=\begin{pmatrix}
1&1&0&0\\
0&1&0&0\\
d&d&1&0\\
0&-k&-1&1
\end{pmatrix}.$$ 
\section{Lagrangian sphere and Lagrangian torus in mirror quintic threefold}
\label{lagrangianST}
In the basis of  homology used in \cite{candelasdelaossa} there are two homology classes which  are supported on Lagrangian submanifolds. We  observe that one class is realized by a Lagrangian 3-sphere and the other by a  Lagrangian 3-torus. \\

Consider the mirror quintic Calabi-Yau threefold $\tilde X_\psi$ associated to the equation \begin{equation*}
\label{ppsi1}
p_{\psi}=z_0^5+z_1^5+z_2^5+z_3^5+z_4^5-5\psi z_0z_1z_2z_3z_4,
\end{equation*} with singularities in $\psi=1, \infty$.
Let $\eta$ be the holomorphic  form on $\tilde X_\psi$. The basis of the matrices $S_\infty$ and $S_0$ are the periods $\int_{\Delta_k} \eta $ with $k=1,2,3,4$. The cycle $\Delta_2$ is a torus associated with the degeneration of the manifold as $\psi$ goes to $\infty$, see \cite[\S 3]{candelasdelaossa}. In coordinates it can be described as
\begin{equation}
\label{cycletorus}
\Delta_2=\{[1:z_1:z_2:z_3:z_4]\in \p^4\text{ }|\text{ } |z_1|=|z_2|=|z_3|=r {\footnotesize \text{ and } z_4 \text{ given by equation $p_{\psi}=0$ when $\psi \to \infty$}}\}
\end{equation}
for $r>0$ small enough, and  $z_4$ is defined as the branch of the solution $p_{\psi}(z)=0$ which tends to zero as $\psi \to \infty$. The cycle $\Delta_2$ does not intersect the curves $C_{ijk}$, and so its quotient by the group $G$ is again a torus.

\begin{proposition}
	The cycle $\Delta_2$ is a Lagrangian submanifold of $( X_\psi, \omega)$, where $\omega$ is the symplectic form given by the pullback of the Fubini-Study form.
\end{proposition}
\begin{proof}
	Consider the Hamiltonian $S^1$-space $(\C^{5}\setminus \{0\}, \omega_{can}, S^1, \mu)$, where $\mu(z)=\frac{-||z||^2+1}{2}$. By Marsden-Weinstein-Meyer theorem, there exist a symplectic form in the reduction  $\mu^{-1}(0)/S^1=\p^4$, and in this case it corresponds to the Fubini-Study form $\omega_{FS}$, see for example \cite[\S 5]{mcduffbook} or \cite[\S 23]{annas}. Furthermore if we denote the reduction by
	\begin{center}
		\begin{tikzpicture}
		\matrix (m) [matrix of math nodes, row sep=1em,
		column sep=2em]{
			\mu^{-1}(0)=S^9 &\C^{5}\setminus\{0\}\\
			\p^4& \\};			
		\path[-stealth]
		(m-1-1) edge [bend right=0] node [above] {$\imath$}  (m-1-2) 
		edge [bend right=0] node [left] {$\pi$} (m-2-1)
		(m-1-2) edge [bend right=0] node [below] {$pr$}  (m-2-1);
		\end{tikzpicture}
	\end{center} 
	the reduced form satisfies $\pi^*\omega_{FS}=\imath^*\omega_{can}$. 
	The canonical form can be written as  $\omega_{can}=\frac{1}{2}\sum_jd|z_j|^2\wedge d\theta_j$. Therefore, for $\epsilon>0$ small enough, the set
	$$T:=\{(z_0, z_1, z_2, z_3, z_4)\in  \C^5\text{ }|\text{ } |z_0|=\epsilon \text{ , }|z_1|=|z_2|=|z_3|=r\text{ , }|z_4|^2=1-\epsilon^2-3r^2\}\subset S^9,$$ is a  Lagrangian submanifold of $(\C^5, \omega_{can}).$ Besides, $\Delta_2$ is the intersection of $ X_{\psi}$ with the projection of $T$ to $\p^4$. Consequently, the tangent space of $\Delta_2$ is contained in the tangent space of $\pi(T)$. Since $0=(\pi^*\omega_{FS})|_{T}=(\omega_{FS})|_{\pi(T)}$, we conclude that $(\omega_{FS})|_{\Delta_2}=0$.
\end{proof}
The cycle  $\Delta_4$ is  associated with the degeneration of the manifold when $\psi$ goes to 1 \cite[\S 3]{candelasdelaossa}. In coordinates can be described as
\begin{equation}
\label{cycletorus}
\Delta_4=\{[1:z_1:z_2:z_3:z_4]\in \p^4\text{ }|\text{ } z_1, z_2, z_3  {\footnotesize \text{ reals and } z_4 \text{ given by equation $p_{\psi}=0$ when $\psi \to 1$}}\}
\end{equation}
where $z_4$ is defined as the branch of of the solution of $p_\psi(z)=0$ which is an $S^3$ when $\psi\to 1$. Follows from the next proposition that $\Delta_4$ is an Lagrangian sphere $S^3$.
\begin{proposition}
	The cycle $\Delta_4$ is a vanishing cycle.
\end{proposition}
\begin{proof}
	 In the chart $z_0=1$ consider the function $f:\C^4\to \C$ given by $f(z_1,\ldots, z_4)=p_{\psi}(1,z_1,\ldots, z_4)$. The critical points of $f$, $(\xi^{k_1} \psi, \xi^{k_2} \psi, \xi^{k_3} \psi, \xi^{k_4} \psi)$ where $\xi=e^{\frac{2\pi i}{5}}$, $k_i=1,\ldots, 5$ and $5|\sum_{j=1}^4 k_j$  are non degenerated. After doing the quotient by the finite group $G$, these critical points are identified with $(\psi,\ldots, \psi)$.
	 
	 For real $\psi>1$ close enough to 1  and by taking $z_j=x_j+iy_j$  we have that the map can be locally  defined as $$f(z_1,...,z_4)=(1-\psi^5)+\sum x_j^2-\sum y_j^2+2i\prod x_jy_j,$$
	 and so the vanishing cycle $\delta_\gamma$  in Proposition \ref{vanishingthimble} is the sphere $(\psi^5-1)=\sum x_j^2$.
\end{proof}
 Let  $\delta_1, \delta_2, \delta_3, \delta_4$ be the basis  on which the matrices $S_1$ and $S_\infty$ are written.
 Consider  the isomorphism between $span\{\delta_i\}_{i=1}^4$ and $\R^4$ with the canonical basis,  given by $\sum_{i=1}^4n_i\delta_i \to (n_1,n_2,n_3,n_4)$. Thus, the monodromy acting on a vector $\delta=\sum_{i=1}^4 n_i \delta_i$ corresponds to 
  \begin{equation*}
  S_j(\delta)=(n_1\hspace{1mm}n_2\hspace{1mm}n_3\hspace{1mm}n_4)S_j\begin{pmatrix}
  \delta_1\\
  \delta_2\\
  \delta_3\\
  \delta_4
  \end{pmatrix}\hspace{4mm}\text{ with }j=1,\infty.
\end{equation*}
  From \cite{candelasdelaossa} and Picard-Lefschetz formula, we know that  
 the monodromy matrices satisfy $S_\infty \Delta_2=\Delta_2$ and $S_1\Delta_2=\Delta_2+\Delta_4$. Therefore $\Delta_2=\delta_2 \equiv [0 \text{ }1\text{ }0\text{ }0]$ and  $\Delta_4=\delta_4 \equiv [0 \text{ }0\text{ }0\text{ }1]$.
\section{Orbits for $\delta_2$ and $\delta_4$}
 Let $H$ be the subgroup of $Sp(4)$, generated by $M_0$ and $M_1$. Moreover, the vectors  $\delta_2=(0 \text{ }1\text{ }0\text{ }0)$ and  $\delta_4=(0 \text{ }0\text{ }0\text{ }1)$ are invariants by the change of basis $M$ defined in (\ref{changeofbasis}). In this section we compute the orbit of $\delta_2$ and $\delta_4$ by the action of $H$ in $\Z_p$, for some prime numbers $p$.\\

For the mirror quintic $\tilde X$, any element in $H_3(\tilde X,\Z)$ which is in  the orbit $H\cdot \delta_4$  is a homology class  supported in a Lagrangian 3-sphere, and any element in the orbit $H\cdot \delta_2$ a homology class supported in a Lagrangian 3-torus.  So far we have not computed the orbits in $\Z$. However, considering $H_3(\tilde X,\Z_p)$ for some primes $p$, it is possible to compute the orbits. The next lemma helps us to reduce the possible words appearing in $H$ mod $p\Z$.
\begin{lemma}
	\begin{align*}
	\mod_p(M_0^{p})=Id_4,&\hspace{2mm}p\neq 2,3, \hspace{2mm} mod_2(M_0^4)=Id_4, mod_3(M_0^9)=Id_4,  \\
	&\mod_p(M_1^p)=Id_4\hspace{4mm}\forall \text{ prime }p.
	\end{align*}
\end{lemma}

\begin{proof}
	Computing the power of theses matrix, we have
	$$M_0^m=\begin{pmatrix}
	1&m&0&0\\
	0&1&0&0\\
	dm&a_m&1&0\\
	b_m&c_m&-m&1
	\end{pmatrix} \text{ and }M_1^m=\begin{pmatrix}
	1&0&0&0\\
	0&1&0&m\\
	0&0&1&0\\
	0&0&0&1
	\end{pmatrix}$$
	where $a_m=\frac{d}{2}m(m+1)\text{ , } b_m=\frac{d}{2}m(1-m)$ and $ c_m=\frac{d}{6}m(1-m^2)-km.$
	Thus, it is enough to show that $p|a_{p}$, $p|b_{p}$ and $p|c_{p}$. However, this is immediate because $2|(p+1)$, $2|(1-p)$ and $6|(1-p^2)$, for $p\neq 2,3$ prime. 
\end{proof} 
Let $v$ be the vector $\delta_2$ (or $\delta_4$) and we denote by $orb_p$ the list of the vector in the orbit of $v$ modulo $p$. Firstly, we start with  $orb_p=\{mod_p(v)\}$ and we compute the  vectors     $\mod_p(wM_1^j M_0^i) \text{  for } j=0,\ldots ,p\text{ , }i=0,\ldots p\text{ and }w\in orb_p$. If these vectors are not in $orb_p$, we add them to $orb_p$. This step is repeated with the new $orb_p$, if there are not new vectors then the orbit is complete. This process is finite because we have at most $p^4$ different vectors in $(\Z/p\Z)^4$.  We summarize the algorithm  \footnote{We have written a MATLAB code for the computation. It is available in  https://github.com/danfelmath/mirrorquintic.git}  to compute the orbit of $v$ modulo $p$  as follows,
 
\begin{tiny}
\begin{algorithm}[H]
	\SetAlgoLined
	\KwIn{$v$, $M_0$, $M_1$, $p$}
	\KwOut{$Orb_p$}
	$Orb_p:=mod_p(v)$
	
	$norm:=1$
	
	\While{$norm>0$}{
		$W:=Orb_p$;
		$L:=length(Orb_p)$
		
		$l:=1$; $c:=1$
		
		\While{$l\leq L$}{
			
			j:=0
			
			\While{$j\leq p$}{
				
				$i:=0$
				
				\While{$i\leq p$}{
					$v_{aux}:=mod_p(W(l)*(M_1^jM_0^i))$
					
					\eIf{$v_{aux}\not\in Orb_p$}{
						
						$Orb_p(L+c):=v_{aux}$
						
						$c:=c+1$; $i:=i+1$
					}
					{$i:=i+1$}
				}
				$j:j+1$
			}
			$l:=l+1$
		}
		$norm:=length(Orb_p)-length(W)$
	}
\end{algorithm}
\end{tiny}

\begin{proof}[Proof of theorem \ref{main2}]
	\label{orbita5}
	Consider the free group $H$ when $d=k=5$. Given $p$, we denote the orbit of $\delta_2$ and $\delta_4$ modulo $p$ as $orb_p(\delta_2)$ and  $orb_p(\delta_4)$, respectively. By using the previous algorithm we have,
	\begin{eqnarray*}
		orb_2(\delta_2)=\{ 
	(0\text{ }0\text{ }1\text{ }1),\text{ } 	
	(0\text{ }1\text{ }0\text{ } 0),\text{ }
	(0\text{ }1\text{ }0\text{ }1),\text{ }
	(1\text{ }     0\text{ }     0\text{ }     1),\text{ }
	(1\text{ }0\text{ }     1\text{ }     1)\}.
	\end{eqnarray*}
	\begin{eqnarray*}
		orb_2(\delta_4)=\{ 
	(0\text{ }0\text{ }0\text{ }1),\text{ } 
	(0\text{ }0\text{ }1\text{ }0),\text{} 	
	(0\text{ }1\text{ }1\text{ }0),\text{}
	(0\text{ }1\text{ }1\text{ }1),\text{}
	(1\text{ }0\text{ }0\text{ }0),\text{}
	 \nonumber\\
	(1\text{ }0\text{ }1\text{ }0),\text{}
	(1\text{ }1\text{ }0\text{ }0),\text{}
	(1\text{ }1\text{ }0\text{ }1),\text{}
	(1\text{ }1\text{ }1\text{ }0),\text{}
	(1\text{ }1\text{ }1\text{ }1)\text{}
	  \}.
	\end{eqnarray*}
	\begin{eqnarray*}
	orb_5(\delta_2)=\{
     (0\text{ }1\text{ }0\text{ }0),\text{}
     (0\text{ }1\text{ }0\text{ }1),\text{}
     (0\text{ }1\text{ }0\text{ }2),\text{}
     (0\text{ }1\text{ }0\text{ }3),\text{}
     (0\text{ }1\text{ }0\text{ }4),\text{}
     	 \nonumber\\
     (0\text{ }1\text{ }1\text{ }0),\text{}
     (0\text{ }1\text{ }1\text{ }1),\text{}
     (0\text{ }1\text{ }1\text{ }2),\text{}
     (0\text{ }1\text{ }1\text{ }3),\text{}
     (0\text{ }1\text{ }1\text{ }4),\text{}
     	 \nonumber\\
     (0\text{ }1\text{ }2\text{ }0),\text{}
     (0\text{ }1\text{ }2\text{ }1),\text{}
     (0\text{ }1\text{ }2\text{ }2),\text{}
     (0\text{ }1\text{ }2\text{ }3),\text{}
     (0\text{ }1\text{ }2\text{ }4),\text{}
     	 \nonumber\\
     (0\text{ }1\text{ }3\text{ }0),\text{}
     (0\text{ }1\text{ }3\text{ }1),\text{}
     (0\text{ }1\text{ }3\text{ }2),\text{}
     (0\text{ }1\text{ }3\text{ }3),\text{}
     (0\text{ }1\text{ }3\text{ }4),\text{}
     	 \nonumber\\
     (0\text{ }1\text{ }4\text{ }0),\text{}
     (0\text{ }1\text{ }4\text{ }1),\text{}
     (0\text{ }1\text{ }4\text{ }2),\text{}
     (0\text{ }1\text{ }4\text{ }3),\text{}
     (0\text{ }1\text{ }4\text{ }4)\}.  
	\end{eqnarray*}
	\begin{eqnarray*}
		orb_5(\delta_4)=\{
	(0 \text{ }0\text{ }0\text{ }1),\text{ }
    (0 \text{ }0\text{ }1\text{ }1),\text{ }
    (0 \text{ }0\text{ }2\text{ }1),\text{ }
    (0 \text{ }0\text{ }3\text{ }1),\text{ }
    (0 \text{ }0\text{ }4\text{ }1)
\}.
	\end{eqnarray*}
	$$orb_p(\delta_2)=orb_p(\delta_4)=(\Z/p\Z)^4\setminus{(0\text{ }0\text{ }0\text{ }0}), \text{ for }p=3,7,11,13,17,19,23. $$
\end{proof}
From the map $H_3(\tilde X, \Z)\xto{mod_p} H_3(\tilde X, \Z_p)$  we have that if $\delta\in H_3(\tilde X, \Z)$ is a primitive class and it is  in the orbit of $\delta_2$ (or $\delta_4$), then  $mod_p(\delta)\in orb_p(\delta_2)$ (or $mod_p(\delta_4)\in orb_p (\delta_4)$) for all $p$.  We think that the converse should be true; that is the Conjecture  \ref{conjecture1}.\\

For the other examples of quintic threefolds appearing in Table \ref{14examplevalues}, we have analogous results. However,  in this case we do not know if the vectors $\delta_2=(0\text{ } 1\text{ } 0\text{ } 0)$ and $\delta_4=(0\text{ } 0\text{ } 0 \text{ }1)$ are really supported in a Lagrangian submanifold. In  Table \ref{orbituv} we present the orbits of the vectors $\delta_2$ and $\delta_4$ modulo $p$ for the fourteen cases of $(d,k )$. If the orbit is $(\Z/p\Z)^4\setminus{(0\text{ }0\text{ }0\text{ }0})$ we call it \textit{complete}. The orbits for the vector $\delta_2$ are  presented in  Table \ref{orbitu} and the orbits for the vector $\delta_4$  are presented in  Table \ref{orbitv}.

\begin{table}[htbp]
	\scriptsize
	\centering
	\begin{tabular}{|l|l|l|}
		\hline
		$(d,k)$ & Prime & Orbit\\
		\hline
		\hline
		$(5,5)$&$p=5$ & 
		$
		(0  0 0 1),
		(0  0 1 1), 
		(0  0 2 1), 
		(0  0 3 1), 
		(0  0 4 1),
		(0 1 0 0),
		(0 1 0 1),
		(0 1 0 2),
		(0 1 0 3),
		(0 1 0 4),$\\&&$
		(0 1 1 0),
		(0 1 1 1),
		(0 1 1 2),
		(0 1 1 3),
		(0 1 1 4),
		(0 1 2 0),
		(0 1 2 1),
		(0 1 2 2),
		(0 1 2 3),
		(0 1 2 4),$\\&&$
        (0 1 3 0),
        (0 1 3 1),
        (0 1 3 2), 
        (0 1 3 3),
        (0 1 3 4),
        (0 1 4 0),
        (0 1 4 1),
        (0 1 4 2),
        (0 1 4 3),
        (0 1 4 4)		
		$
		
		\\
		\hline
		
		&$p=2,3,7,11,13,17,19,23$ & Complete\\
		
		\hline
		$(2,4)$&$p=2$ & 
		$
	     (0     0     0     1),
	     (0     0     1     1),
	     (0     1     0     0),
	     (0     1     0     1),
	     (0     1     1     0),
	     (0     1     1     1)$\\
		\hline
		&$p=3,5,7,11,13,17,19,23$& Complete\\
		\hline
		
		$(1,4)$	
		&$p=2,3,5,7,11,13,17,19,23$& Complete\\
		\hline

		$(16,8)$&$p=2$ & 
	   $(0     0     0     1),
		(0     0     1     1),
		(0     1     0     0),
		(0     1     0     1),
		(0     1     1     0),
    	(0     1     1     1)$\\
		\hline
		&$p=3,5,7,11,13,17,19,23$& Complete\\
		\hline

		$(12,7)$&$p=2$&	   
	   $(0     0     0     1),
		(0     0     1     1),
		(0     1     0     0),
		(0     1     0     1),
		(0     1     1     0),
		(0     1     1     1)$\\
		\hline
		&$p=3$& 
		$ 
		     (0     0     0     1),
		     (0     0     0     2),
		     (0     0     2     1),
		     (0     0     2     2),
		     (0     1     0     0),
		     (0     1     0     1),
		     (0     1     0     2),
		     (0     1     1     0),$\\&&$
		     (0     1     1     1),
		     (0     1     1     2),
		     (0     2     1     0),
		     (0     2     1     1),
		     (0     2     1     2),
		     (0     2     2     0),
		     (0     2     2     1),
		     (0     2     2     2)
		$\\
		\hline
		&$p=5,7,11,13,17,19,23$& Complete\\
		\hline

		$(8, 6)$&$p=2$ & $
		(0     0     0     1),
		(0     0     1     1),
		(0     1     0     0),
		(0     1     0     1),
		(0     1     1     0),
		(0     1     1     1)$\\
		\hline
		&$p=3,5,7,11,13,17,19,23$& Complete\\
		\hline	
		$(4,5)$&$p=2$ & $
		(0     0     0     1),
		(0     0     1     1),
		(0     1     0     0),
		(0     1     0     1),
		(0     1     1     0),
		(0     1     1     1)$\\
		\hline
		&$p=3,5,7,11,13,17,19,23$& Complete\\
		\hline	
		$(2,3)$&$p=2$ & $
		(0     0     0     1),
		(0     0     1     1),
		(0     1     0     0),
		(0     1     0     1),
		(0     1     1     0),
		(0     1     1     1)$\\
		\hline
		&$p=3,5,7,11,13,17,19,23$& Complete\\
		\hline		
		$(1,2)$ 
		&$p=2,3,5,7,11,13,17,19,23$& Complete\\
		\hline	
		$(6,5)$&$p=2$ & $
		(0     0     0     1),
		(0     0     1     1),
		(0     1     0     0),
		(0     1     0     1),
		(0     1     1     0),
		(0     1     1     1)$\\
		\hline
		&$p=3$& 
		$
      (0     0     0     1),
      (0     0     0     2),
      (0     0     1     1),
      (0     0     1     2),
      (0     1     0     0),
      (0     1     0     1),
      (0     1     0     2),
      (0     1     2     0),$\\&&$
      (0     1     2     1),
      (0     1     2     2),
      (0     2     1     0),
      (0     2     1     1),
      (0     2     1     2),
      (0     2     2     0),
      (0     2     2     1),
      (0     2     2     2)
		$\\
		\hline
		&$p=5,7,11,13,17,19,23$& Complete\\
		\hline		
		$(3,4)$&$p=3$& 		$ 
		(0     0     0     1),
		(0     0     0     2),
		(0     0     2     1),
		(0     0     2     2),
		(0     1     0     0),
		(0     1     0     1),
		(0     1     0     2),
		(0     1     1     0),$\\&&$
		(0     1     1     1),
		(0     1     1     2),
		(0     2     1     0),
		(0     2     1     1),
		(0     2     1     2),
		(0     2     2     0),
		(0     2     2     1),
		(0     2     2     2)
		$\\
		\hline
		&$p=2,5,7,11,13,17,19,23$& Complete\\
		\hline
		$(1,3)$
		&$p=2,3,5,7,11,13,17,19,23$& Complete\\
		\hline
		
		$(4,4)$&$p=2$ & $	
		(0     0     0     1),
		(0     0     1     1),
		(0     1     0     0),
		(0     1     0     1),
		(0     1     1     0),
		(0     1     1     1)$\\
		\hline
		&$p=3,5,7,11,13,17,19,23$& Complete\\
		\hline

		$(9,6)$&$p=3$& $
          (0     0     0     1),
          (0     0     1     1),
          (0     0     2     1),
          (0     1     0     0),
          (0     1     0     1),
          (0     1     0     2),
          (0     1     1     0),
          (0     1     1     1),$\\&&$
          (0     1     1     2),
          (0     1     2     0),
          (0     1     2     1),
          (0     1     2     2)$\\
		\hline
		&$p=2,5,7,11,13,17,19,23$& Complete\\
		\hline
	\end{tabular}
	\caption{Orbit of  vectors $\delta_2$ and $\delta_4$  by the monodromy action for the fourteen mirror Calabi-Yau threefolds.}
	\label{orbituv}
\end{table}

\begin{table}[htbp]
	\scriptsize
	\centering
	\begin{tabular}{|l|l|l|}
		\hline
		$(d,k)$ & Prime & Orbit\\
		\hline
		\hline
		$(5,5)$&$p=2$ & 
		$
     (0     0     1     1),
     (0     1     0     0),
     (0     1     0     1),
     (1     0     0     1),
     (1     0     1     1)$\\
		\hline
		&$p=5$ &
		$
     (0     1     0     0),
     (0     1     0     1),
     (0     1     0     2),
     (0     1     0     3),
     (0     1     0     4),
     (0     1     1     0),
     (0     1     1     1),
     (0     1     1     2),
     (0     1     1     3),
     (0     1     1     4),$\\&&$
     (0     1     2     0),
     (0     1     2     1),
     (0     1     2     2),
     (0     1     2     3),
     (0     1     2     4),
     (0     1     3     0),
     (0     1     3     1),
     (0     1     3     2),
     (0     1     3     3),
     (0     1     3     4),$\\&&$
     (0     1     4     0),
     (0     1     4     1),
     (0     1     4     2),
     (0     1     4     3),
     (0     1     4     4)$\\
		\hline	
		&$p=3,7,11,13,17,19,23$ & Complete\\
		\hline
		$(2,4)$&$p=2$ & $
	     (0     1     0     0),
	     (0     1     0     1)
	     (0     1     1     0),
	     (0     1     1     1)$\\
		\hline
		&$p=3,5,7,11,13,17,19,23$ & Complete\\
		\hline
		$(1,4)$&$p=2$&   
		$
     (0     0     1     0),
     (0     1     0     0),
     (0     1     0     1),
     (0     1     1     0),
     (0     1     1     1),
     (1     0     0     1),
     (1     0     1     0),
     (1     1     1     0),
     (1     1     1     1)$\\
		\hline
		&$p=3,5,7,11,13,17,19,23$ & Complete\\
		\hline	
		$(16,8)$&$p=2$ & $  
      (0     1     0     0),
      (0     1     0     1),
      (0     1     1     0),
      (0     1     1     1)$\\
		\hline
		&$p=3,5,7,11,13,17,19,23$ & Complete\\
		\hline	
		$(12,7)$&$p=2$&$ 
       (0     0     1     1),
       (0     1     0     0),
       (0     1     0     1)
$
		\\
		\hline
		&$p=3$& $   
     (0     0     2     1),
     (0     0     2     2),
     (0     1     0     0),
     (0     1     0     1),
     (0     1     0     2),
     (0     2     1     0),
     (0     2     1     1),
     (0     2     1     2)
$\\
		
		\hline
		&$p=5,7,11,13,17,19,23$ & Complete\\
		\hline			
		$(8, 6)$&$p=2$ & $
     (0     1     0     0),
     (0     1     0     1),
     (0     1     1     0),
     (0     1     1     1)
$\\
		\hline
		&$p=3,5,7,11,13,17,19,23$ & Complete\\
		\hline		
		$(4,5)$&$p=2$ & $ 
        (0     0     1     1),
        (0     1     0     0),
        (0     1     0     1)
$\\
		\hline
		&$p=3,5,7,11,13,17,19,23$ & Complete\\
		\hline		
		$(2,3)$&$p=2$ & $   
      (0     0     1     1),
      (0     1     0     0),
      (0     1     0     1)$\\
		\hline
		&$p=3,5,7,11,13,17,19,23$ & Complete\\
		\hline	
		
		$(1,2)$&$p=2$&$
       (0     0     1     0),
       (0     1     0     0),
       (0     1     0     1),
       (0     1     1     0),
       (0     1     1     1),
       (1     0     0     1),
       (1     0     1     0),
       (1     1     1     0),
       (1     1     1     1)$\\
		\hline
		&$p=3,5,7,11,13,17,19,23$ & Complete\\
		\hline			
		$(6,5)$&$p=2$ & $  
       (0     0     1     1),
       (0     1     0     0),
       (0     1     0     1)$\\
		\hline
		&$p=3$& 
		$
	 (0     0     1     1),
	 (0     0     1     2),
	 (0     1     0     0),
	 (0     1     0     1),
	 (0     1     0     2),
	 (0     2     2     0),
	 (0     2     2     1),
	 (0     2     2     2)$\\
		\hline
		&$p=5,7,11,13,17,19,23$ & Complete\\
		\hline	
		
		$(3,4)$&$p=2$& $
          (0     0     1     0),
          (0     1     0     0),
          (0     1     0     1),
          (0     1     1     0),
          (0     1     1     1),
          (1     0     0     1),
          (1     0     1     0),
          (1     1     1     0),
          (1     1     1     1)$\\
		\hline
		&$p=3$&$
		(0     0     2     1),
		(0     0     2     2),
		(0     1     0     0),
		(0     1     0     1),
		(0     1     0     2),
		(0     2     1     0),
		(0     2     1     1),
		(0     2     1     2)$\\
		\hline
		&$p=5,7,11,13,17,19,23$ & Complete\\
		\hline	
		
		$(1,3)$&$p=2$&$
         (0     0     1     1),
         (0     1     0     0),
         (0     1     0     1),
         (1     0     0     1),
         (1     0     1     1)        
$\\
		\hline
		&$p=3,5,7,11,13,17,19,23$ & Complete\\
		\hline	
		
		$(4,4)$&$p=2$ & $
      (0     1     0     0),
      (0     1     0     1),
      (0     1     1     0),
      (0     1     1     1)$\\
		\hline
		&$p=3,5,7,11,13,17,19,23$ & Complete\\
		\hline	
		
		$(9,6)$&$p=2$& $
     (0     0     1     0),
     (0     1     0     0),
     (0     1     0     1),
     (0     1     1     0),
     (0     1     1     1),
     (1     0     0     1),
     (1     0     1     0),
     (1     1     1     0),
     (1     1     1     1)$\\
		\hline
		&$p=3$&$
     (0     1     0     0),
     (0     1     0     1),
     (0     1     0     2),
     (0     1     1     0),
     (0     1     1     1),
     (0     1     1     2),
     (0     1     2     0),
     (0     1     2     1),
     (0     1     2     2)$\\
		\hline
		&$p=5,7,11,13,17,19,23$ & Complete\\
		\hline	
	\end{tabular}
	\caption{Orbit of  vector $\delta_2$  by the monodromy action for the fourteen mirror Calabi-Yau threefolds.}
	\label{orbitu}
\end{table}

\begin{table}[htbp]
	\scriptsize
	\centering
	\begin{tabular}{|l|l|l|}
		\hline
		$(d,k)$ & Prime & Orbit\\
		\hline
		\hline
		$(5,5)$&$p=2$ & 
		$
     (0     0     0     1),
     (0     0     1     0),
     (0     1     1     0),
     (0     1     1     1),
     (1     0     0     0),
     (1     0     1     0),
     (1     1     0     0),
     (1     1     0     1),
     (1     1     1     0),
     (1     1     1     1)$\\
		\hline
		&$p=5$ &
		$
		(0     0     0     1),
		(0     0     1     1),
		(0     0     2     1),
		(0     0     3     1),
		(0     0     4     1)$\\
		\hline
		&$p=3,7,11,13,17,19,23$ & Complete\\
		
		\hline
		$(2,4)$&$p=2$ & $ 
     (0     0     0     1),
     (0     0     1     1)$\\
		\hline
		&$p=3,5,7,11,13,17,19,23$ & Complete\\
		\hline
		
		$(1,4)$&$p=2$&   
		$
        (0     0     0     1),
        (0     0     1     1),
        (1     0     0     0),
        (1     0     1     1),
        (1     1     0     0),
        (1     1     0     1)$\\
		\hline
		&$p=3,5,7,11,13,17,19,23$ & Complete\\
		\hline	
		$(16,8)$&$p=2$ & $  
         (0     0     0     1),
         (0     0     1     1)$\\
		\hline
		&$p=3,5,7,11,13,17,19,23$ & Complete\\
		\hline
		$(12,7)$&$p=2$&$ 
     (0     0     0     1),
     (0     1     1     0),
     (0     1     1     1)$
		\\
		\hline
		&$p=3$&$
          (0     0     0     1),
          (0     0     0     2),
          (0     1     1     0),
          (0     1     1     1),
          (0     1     1     2),
          (0     2     2     0),
          (0     2     2     1),
          (0     2     2     2)$\\
		\hline
		&$p=5,7,11,13,17,19,23$ & Complete\\
		\hline	
		$(8, 6)$&$p=2$ & $
		   (0     0     0     1),
		   (0     0     1     1)$\\
		\hline
		&$p=3,5,7,11,13,17,19,23$ & Complete\\
		\hline
		$(4,5)$&$p=2$ & $ 
		   (0     0     0     1),
		   (0     1     1     0),
		   (0     1     1     1)$\\
		\hline
		&$p=3,5,7,11,13,17,19,23$ & Complete\\
		\hline	
		$(2,3)$&$p=2$ & $   
     (0     0     0     1),
     (0     1     1     0),
     (0     1     1     1)$\\
		\hline
		&$p=3,5,7,11,13,17,19,23$ & Complete\\
		\hline	
		$(1,2)$&$p=2$&$
         (0     0     0     1),
         (0     0     1     1),
         (1     0     0     0),
         (1     0     1     1),
         (1     1     0     0),
         (1     1     0     1)$\\
		\hline
		&$p=3,5,7,11,13,17,19,23$ & Complete\\
		\hline
		
		$(6,5)$&$p=2$ & $  
         (0     0     0     1),
         (0     1     1     0),
         (0     1     1     1)
$\\
		\hline
		&$p=3$& 
		$ 
        (0     0     0     1),
        (0     0     0     2),
        (0     1     2     0),
        (0     1     2     1),
        (0     1     2     2),
        (0     2     1     0),
        (0     2     1     1),
        (0     2     1     2)$\\
		\hline
		&$p=5,7,11,13,17,19,23$ & Complete\\
		\hline
		$(3,4)$&$p=2$& $
	     (0     0     0     1),
	     (0     0     1     1),
	     (1     0     0     0),
	     (1     0     1     1),
	     (1     1     0     0),
	     (1     1     0     1)
		$\\
		\hline
		&$p=3$&$
     (0     0     0     1),
     (0     0     0     2),
     (0     1     1     0),
     (0     1     1     1),
     (0     1     1     2),
     (0     2     2     0),
     (0     2     2     1),
     (0     2     2     2)
$\\
		\hline
		&$p=5,7,11,13,17,19,23$ & Complete\\
		\hline	
		$(1,3)$&$p=2$&$
	     (0     0     0     1),
	     (0     0     1     0),
	     (0     1     1     0),
	     (0     1     1     1),
	     (1     0     0     0),
	     (1     0     1     0),
	     (1     1     0     0),
	     (1     1     0     1),
	     (1     1     1     0),
	     (1     1     1     1)
		$\\
		\hline
		&$p=3,5,7,11,13,17,19,23$ & Complete\\
		\hline	
		$(4,4)$&$p=2$ & $
	     (0     0     0     1),
	     (0     0     1     1)
		$\\
		\hline
		&$p=3,5,7,11,13,17,19,23$ & Complete\\
		\hline
		
		$(9,6)$&$p=2$& $
      (0     0     0     1),
      (0     0     1     1),
      (1     0     0     0),
      (1     0     1     1),
      (1     1     0     0),
      (1     1     0     1),
		$\\
		\hline
		&$p=3$&$
       (0     0     0     1),
       (0     0     1     1),
       (0     0     2     1)
		$\\
		\hline
		&$p=5,7,11,13,17,19,23$ & Complete\\
		\hline
		
	\end{tabular}
	\caption{ Orbit of  vector $\delta_4$  by the monodromy action for the fourteen mirror Calabi-Yau threefolds.}
	\label{orbitv}
\end{table}

\clearpage

\begin{footnotesize}
	\bibliographystyle{abbrv}
   \bibliography{References}
\end{footnotesize}

\bigskip

\sf{\noindent Daniel L\'opez Garcia\\
	Instituto de Matematica Pura e Aplicada (IMPA),  \\ 
	Estrada Dona Castorina 110, Rio de Janeiro, 22460-320, RJ, Brazil.\\
	daflopez@impa.br}

\end{document}